\documentclass[reqno, 10pt]{amsart}
\usepackage{amssymb}
\usepackage{amsmath}
\usepackage[mathscr]{euscript}
\usepackage{hyperref}
\usepackage{graphicx}
\usepackage[normalem]{ulem}

\usepackage{amssymb}
\usepackage{hyperref}
\RequirePackage[dvipsnames]{xcolor} 
\definecolor{halfgray}{gray}{0.55} 
\definecolor{webgreen}{rgb}{0,0.5,0}
\definecolor{webbrown}{rgb}{.6,0,0} \hypersetup{%
  colorlinks=true, linktocpage=true, pdfstartpage=3,
  pdfstartview=FitV,%
  breaklinks=true, pdfpagemode=UseNone, pageanchor=true,
  pdfpagemode=UseOutlines,%
  plainpages=false, bookmarksnumbered, bookmarksopen=true,
  bookmarksopenlevel=1,%
  hypertexnames=true,
  pdfhighlight=/O,
  urlcolor=webbrown, linkcolor=RoyalBlue,
  citecolor=webgreen, 
  pdftitle={},%
  pdfauthor={},%
  pdfsubject={2000 MAthematical Subject Classification: Primary:},%
  pdfkeywords={},%
  pdfcreator={pdfLaTeX},%
  pdfproducer={LaTeX with hyperref}%
}

\usepackage{enumitem}

\newtheorem{theorem}{Theorem}[section]

\newtheorem{corollary}[theorem]{Corollary}
\newtheorem{proposition}[theorem]{Proposition}

\theoremstyle{definition}

\newtheorem{remark}[theorem]{Remark}

\def\R{\mathbb{R}}
\def\N{\mathbb{N}}
\def\Reg{\mathcal{R}^{\mu}}
\newcommand{\norm}[1]{{\left\lVert \, #1 \, \right\rVert}}

\def\Id{\text{Id}}

\begin{document}

\title
[A remark on the invertibility of semi-invertible cocycles]
{A remark on the invertibility of semi-invertible cocycles}

\author{Lucas Backes}

\address{Departamento de Matem\'atica, Universidade Federal do Rio Grande do Sul, Av. Bento Gon\c{c}alves 9500, CEP 91509-900, Porto Alegre, RS, Brazil.}
\email{lhbackes@impa.br }

\date{\today}


\begin{abstract}
We observe that under certain conditions on the Lyapunov exponents a semi-invertible cocycle is, indeed, invertible. As a consequence, if a semi-invertible cocycle generated by a H\"{o}lder continuous map $A:M\to M(d, \mathbb{R})$ over a hyperbolic map $f:M\to M$ satisfies a Liv\v{s}ic's type condition, that is, if $A(f^{n-1}(p))\cdot\ldots \cdot A(f(p))A(p)=\Id$ for every $p\in \text{Fix}(f^n)$ then the cocycle is invertible, meaning that $A(x)\in GL(d,\mathbb{R})$ for every $x\in M$, and a Liv\v{s}ic's type theorem is satisfied.
\end{abstract}

\maketitle

\section{Introduction}

Linear cocycles are classical objects in the fields of Dynamical Systems and Ergodic Theory. As a simple example we can cite the derivative of a smooth map. This example also reveals the importance of these objects: it is, for instance, in the core of the study of hyperbolic dynamics and its variations. Nevertheless, the notion of linear cocycle is much broader than that and includes, for instance, stochastic processes and random matrices and arises naturally in many other contexts like in the spectral theory of Schr\"odinger operators.

In the present work we are interested in a particular type of linear cocycles, namely, the semi-invertible ones. Given a homeomorphism $f:M\to M$ acting on a compact metric space $(M,d)$ and a measurable matrix-valued map $A:M\rightarrow M(d, \mathbb{R})$, the pair $(A,f)$ is called a \textit{semi-invertible linear cocycle}. Sometimes one calls semi-invertible linear cocycle (over $f$ generated by $A$), instead, the sequence $\lbrace A^n\rbrace _{n\in \mathbb{N}}$ defined by
\begin{equation*}\label{def:cocycles}
A^n(x)=
\left\{
	\begin{array}{ll}
		A(f^{n-1}(x))\ldots A(f(x))A(x)  & \mbox{if } n>0 \\
		Id & \mbox{if } n=0 \\
	\end{array}
\right.
\end{equation*}
for all $x\in M$. The word `semi-invertible' refers to the fact that the action of the underlying dynamical system $f$ is invertible while the action on the fibers given by $A$ may fail to be invertible. Whenever the map $A$ takes values on $GL(d,\R)$, that is, $A:M \to GL(d,\R)$, we call the cocycle generated by $A$ over $f$ an \emph{invertible cocycle}.

The theories of these two classes of objects share many similarities but they also may exhibit very different behaviors. For instance, it was proved by Cao in \cite{Cao03} that if an invertible cocycle $(A,f)$ has only positive Lyapunov exponents with respect to every $f$-invariant probability measure, then it is a uniformly expanding cocycle. On the other hand, in \cite{Bac17} we exhibited an example showing that this is no longer true in the semi-invertible setting: there exists a semi-invertible cocycle $(A,f)$ whose Lyapunov exponents are all larger then a certain constant $c>0$ and a point $x\in M$ for which $A(x)\notin GL(d,\R)$ and, in particular, $(A,f)$ can not be uniformly expanding.

Bearing in mind the relation between these two theories, this note was mainly motivated by the following problem: it was proved by Kalinin \cite{Kal11} that whenever $A:M\to GL(d,\R)$ is a H\"older map and $f$ exhibits enough hyperbolicity, if 
\begin{equation}\label{eq: poc}
A^n(p)=\Id \text{ for every } p\in \text{Fix}(f^n) \text{ and } n\in\N,
\end{equation}
then $A$ is a \emph{coboundary}, that is, there exists a H\"older map $P:M\to GL(d,\R)$ satisfying
$$A(x)=P(f(x))P(x)^{-1} \text{ for every } x\in M.$$
So, a natural question that arises is what happens if instead of an invertible cocycle we have a \emph{semi-invertible} cocycle satisfying \eqref{eq: poc}. It turns out that we can get a similar conclusion (see Section \ref{sec: statements} for precise statements):
\begin{theorem}
Suppose $A:M\to M(d,\R)$ is a H\"older map and $f:M\to M$ exhibits enough hyperbolicity. If the cocycle $(A,f)$ satisfies \eqref{eq: poc} then $A(x)\in GL(d,\R)$ for every $x\in M$ and there exists a H\"older map $P:M\to GL(d,\R)$ satisfying
$$A(x)=P(f(x))P(x)^{-1} \text{ for every } x\in M.$$
\end{theorem}
The previous result is, in fact, a consequence of a more general one that we got which can be roughly stated as (see Theorem \ref{theo:main})
\begin{theorem}
If the semi-invertible cocycle $(A,f)$ has only small Lyapunov exponents, in modulus, then it is, indeed, invertible.
\end{theorem}
In particular, the invertibility of the cocycle $(A,f)$ can be read out of its asymptotic behaviour.

\section{Statements}\label{sec: statements}

Let $(M,d)$ be a compact metric space, $f: M \to M $ a homeomorphism and $A:M\to M(d,\R)$ a $\alpha$-H\"{o}lder continuous map. This means that there exists a constant $C_1>0$ such that
\begin{displaymath}
\norm{A(x)-A(y)} \leq C_1 d(x,y)^{\alpha}
\end{displaymath}
for all $x,y\in M$ where $\norm{A}$ denotes the operator norm of a matrix $A$, that is, $\norm{A} =\sup \lbrace \norm{Av}/\norm{v};\; \norm{v}\neq 0 \rbrace$. 

\subsection{Lyapunov exponents} Given an ergodic $f$-invariant Borel propability measure $\mu$, it was proved in \cite{FLQ10} that there exists a full $\mu$-measure set $\Reg \subset M$, whose points are called $\mu$-regular points, such that for every $x\in \Reg$ there exist numbers $\lambda _1>\ldots > \lambda _{l}\geq -\infty$, called \textit{Lyapunov exponents}, and a direct sum decomposition $\mathbb{R}^d=E^{1,A}_{x}\oplus \ldots \oplus E^{l,A}_{x}$ into vector subspaces which are called \textit{Oseledets subspaces} and depend measurable on $x$ such that, for every $1\leq i \leq l$,
\begin{itemize}
\item dim$(E^{i,A}_{x})$ is constant,
\item $A(x)E^{i,A}_{x}\subseteq E^{i,A}_{f(x)}$ with equality when $\lambda _i>-\infty$
\end{itemize}
and
\begin{itemize}
\item $\lambda _i =\lim _{n\rightarrow +\infty} \dfrac{1}{n}\log \parallel A^n(x)v\parallel$ 
for every non-zero $v\in E^{i,A}_{x}$.
\end{itemize}
This result extends a famous theorem due to Oseledets \cite{Ose68} known as the \textit{multiplicative ergodic theorem} which was originally stated in both, \textit{invertible} (both $f$ and the matrices are assumed to be invertible) and \textit{non-invertible} (neither $f$ nor the matrices are assumed to be invertible) settings (see also \cite{LLE}). While in the invertible case the conclusion is similar to the conclusion above (except that all Lyapunov exponents are finite), in the non-invertible case, instead of a direct sum decomposition into invariant vector subspaces, one only gets an invariant \textit{filtration} (a sequence of nested subspaces) of $\R ^d$. We denote by 
$$\gamma _1(A,\mu)\geq \gamma _2(A,\mu)\geq \ldots \geq \gamma _d(A,\mu)$$
the Lyapunov exponents of $(A,f)$ with respect to the measure $\mu$ counted with multiplicities.

\subsection{Periodic exponential specification property and the Anosov Closing property}
We say that $f$ satisfies the \emph{periodic exponential specification property} if there exists $\theta >0$ so that for every $\delta>0$ there exists $S=S (\delta)>0$ so that for any $x\in M$ and any $n\in \mathbb{N}$ there exists a periodic point $p\in M$ such that $f^{n+S}(p)=p$ satisfying  
$$d (f^{j} (p),f^j (x))<\delta e^{-\theta \min\lbrace j, n-j\rbrace} \text{ for all } j=0,1,\ldots,n.$$
Notice that this is particular version of the sometimes called \emph{Bowen's exponential specification property} where one requeries that the previous property holds for unions of segments of orbits. That is, given points $\{x_1, \ldots ,x_k \}$ in $M$ and any integers $a_1\leq b_1<a_2\leq b_2< \ldots <a_k \leq b_k $ satisfying $a_{j}-b_{j-1}\geq S$, for every $t\geq b_k-a_1+S$ there exists a periodic point $p_t\in M$ so that $f^t(p_t)=p_t$ and 
$$d (f^{j} (p_t),f^j (x_i))<\delta e^{-\theta \min\lbrace j-a_i, b_i-j\rbrace} \text{ for all } a_i\leq j\leq b_i, \text{ and } i=1,\ldots ,k.$$ 
See \cite{Tia17}, for instance. Since we are not going to need this full version, we consider the previous simpler form.

We say that $f$ satisfies the \textit{Anosov Closing property} if there exist $C_2 ,\varepsilon _0 ,\theta >0$ such that if $z\in M$ satisfies $d(f^n(z),z)<\varepsilon _0$ then there exists a periodic point $p\in M$ such that $f^n(p)=p$ and
\begin{displaymath}
d(f^j(z),f^j(p))\leq C_2 e^{-\theta \min\lbrace j, n-j\rbrace}d(f^n(z),z) \text{ for every } j=0,1,\ldots ,n
\end{displaymath}
and a point $y\in M$ satisfying
\begin{displaymath}
d(f^j(y),f^j(p))\leq C_2 e^{-\theta j}d(f^n(z),z) \text{ and } d(f^j(z),f^j(y))\leq C_2 e^{-\theta (n-j)}d(f^n(z),z)
\end{displaymath}
for every $j=0,1,\ldots ,n$ (see \cite{Kal11}).

Every topologically mixing locally maximal hyperbolic set has the periodic exponential specification property as well as the Anosov closing property (see for instance \cite{KH95, Tia17}). In particular, transitive Anosov diffeomorphims satisfy both properties. Moreover, it is easy to see that if a homeomorphism $f$ is topologically conjugated to a map $g$ satisfying both of the previous properties and the conjugacy and its inverse are H\"older continuous then $f$ itself satisfies both properties. In particular, it follows from \cite{Gog10} that there are non-uniformly hyperbolic systems satisfying the periodic exponential specification property and the Anosov closing property.

\subsection{Main results} The main result of this note is the following

\begin{theorem}\label{theo:main}
Let $f:M\to M$ be a homeomorphism satisfying the periodic exponential specification property and $A:M \to M(d,\R) $ a $\alpha$-H\"{o}lder continuous map. Assume there exist $\rho, \tau >0$ satisfying $\rho +\tau <\frac{\alpha\theta}{2}$ so that 
\begin{equation}\label{eq: hip}
-\tau \leq \gamma _d(A,\mu)+ \ldots + \gamma_1(A,\mu)\leq \rho
\end{equation}
for every ergodic $f$-invariant measure $\mu$ on $M$. Then, $A(x)\in GL(d,\R)$ for every $x\in M$.
\end{theorem}

A simple observation is that hypothesis \eqref{eq: hip} is satisfied, for instance, whenever
\begin{equation*}
\frac{-\tau}{d} \leq \gamma _d(A,\mu)\leq \ldots \leq \gamma_1(A,\mu)\leq \frac{\rho}{d}
\end{equation*}
for every ergodic $f$-invariant measure $\mu$. In particular, if all Lyapunov exponents of $(A,f)$ are in a small neighborhood of zero then the semi-invertible cocycle $(A,f)$ is, indeed, invertible. Moreover, under the additional assumption that $f$ satisfies the Anosov closing property, assumption \eqref{eq: hip} can be replaced by the assumption that
$$-\tau \leq \gamma _d(A,\mu_p)+ \ldots + \gamma_1(A,\mu_p)\leq \rho$$
for every ergodic $f$-invariant measure $\mu_p$ supported on periodic points. In fact, by Theorem 2.1 of \cite{Bac17} they are equivalent.

\begin{remark}\label{rem: invertibility a.e.}
Observe that assuming that all Lyapunov exponents of $(A,f)$ are uniformly bounded by below, that is, $\gamma_d(A,\mu)>c$ for every ergodic $f$-invariant measure $\mu$ on $M$ and some $c\in \R$ implies that $A(x)\in GL(d,\R)$ for $\mu$-almost every $x\in M$ with respect to \emph{every} $f$-invariant measure $\mu$ (see Corollary 1 of \cite{Bac17}). On the other hand, as observed in Example 2.2 of \cite{Bac17} and already mentioned in the introduction, this \emph{does not} imply, in general, that $A(x)\in GL(d,\R)$ for every $x\in M$. In fact, in the aforementioned example we exhibited a cocycle $(A,f)$ satisfying
$$0<c \leq \gamma _d(A,\mu)\leq \ldots \leq \gamma_1(A,\mu)\leq \log \left( \max_{x\in M}\|A(x)\| \right) $$
for every ergodic $f$-invariant measure $\mu$ so that $A(x)\notin GL(d,\R)$ for some $x\in M$. Moreover, the map $f$ in that example satisfies the assumptions of Theorem \ref{theo:main}.
\end{remark}

As an interesting consequence of the previous result we get a Liv\v{s}ic's type theorem in the, a priori, semi-invertible setting.

\begin{corollary}
Let $f:M\to M$ be a topologically transitive homeomorphim satisfying the Anosov closing property and the periodic exponential specification property and $A:M \to M(d,\R) $ a $\alpha$-H\"{o}lder continuous map. Assume that 
$$A^n(p)=\Id \text{ for every } p\in \text{Fix}(f^n) \text{ and } n\in \N.$$
Then, there exists a $\alpha$-H\"{o}lder continuous map $P:M\to GL(d,\R)$ so that
$$A(x)=P(f(x))P(x)^{-1} \text{ for every } x\in M.$$
\end{corollary}
\begin{proof}
Since $A^n(p)=\Id$ for every $p\in \text{Fix}(f^n)$ and $n\in \N$, it follows from the comments after our main result that the hypotheses of Theorem \ref{theo:main} are satisfied. In particular, $A(x)\in GL(d,\R)$ for every $x\in M$. Let $C=\max_{x\in M} \{ \norm{A(x)},\norm{A(x)}^{-1} \}$. Then,
\begin{displaymath}
\begin{split}
\norm{A(x)^{-1}-A(y)^{-1}}&= \norm{A(x)^{-1}(A(y)-A(x))A(y)^{-1}}\\
&\leq \norm{A(x)^{-1}}\norm{A(y)^{-1}}\norm{A(x)-A(y)}\\
&\leq C^2 C_1d(x,y)^\alpha.
\end{split}
\end{displaymath}
In particular, $A:M\to GL(d,\R)$ is a $\alpha$-H\"{o}lder continuous map with respect to the distance given by
$$\tilde{d}(A(x),A(y))=\norm{A(x)-A(y)}+\norm{A(x)^{-1}-A(y)^{-1}}$$
on $GL(d,\R)$. Now, the result follows from \cite[Theorem 1.1]{Kal11}. 
\end{proof}

\section{Proof of the main result}

Let $f:M\to M$ be a homeomorphism satisfying the periodic exponential specification property and $A:M \to M(d,\R) $ a $\alpha$-H\"{o}lder continuous map. We start with an auxiliary result.  

\begin{proposition}\label{prop: cresc}
Suppose there exists $\rho>0$ so that 
$$ \gamma_1(A,\mu)\leq \rho$$
for every ergodic $f$-invariant measure $\mu$ on $M$. Then, for every $\varepsilon>0$ there exists a constant $C_\varepsilon>0$ so that
\begin{equation*}
\|A^n(x)\|\leq C_\varepsilon e^{(\rho+\varepsilon) n} 
\end{equation*}
for every $x\in M$ and $n\in \mathbb{N}$.
\end{proposition}
\begin{proof} 
Let us consider $\tilde{A}:M\to M(d+1,\R)$ given by 
\begin{displaymath}
\tilde{A}(x)=\left( \begin{array}{cc}
1 & 0 \\
0 & A(x)\\
\end{array}\right).
\end{displaymath}
Observe that the cocycle $(\tilde{A},f)$ also satisfies $\gamma _1(\tilde{A},\mu )\leq \rho$ for every ergodic $f$-invariant measure $\mu$ on $M$. Moreover, $\|\tilde{A}^n(x)\|\neq 0$ for every $x\in M$ and $n\in \N$. Thus, by Proposition 4.9 of \cite{KS13} applied to $$a_n(x)=\log \|\tilde{A}^n(x)\| -(\rho +\varepsilon)n$$ 
it follows that there exists $N\in \N$ so that $a_N(x)<0$ for every $x\in M$. That is,
$$\|\tilde{A}^N(x)\|\leq e^{(\rho +\varepsilon)N} \text{ for every } x\in M.$$
Consequently, taking $C_\varepsilon=\sup_{j=0,1,\ldots,N}\{\sup_{x\in M}\|\tilde{A}^j(x)\|\}$ and using the submultiplicativity of the norm it follows that
\begin{equation*}
\|\tilde{A}^n(x)\|\leq C_\varepsilon e^{(\rho+\varepsilon) n} 
\end{equation*}
for every $x\in M$ and $n\in \mathbb{N}$. Now, observing that $\|A^n(x)\|\leq \|\tilde{A}^n(x)\|$ for every $x\in M$ and $n\in \mathbb{N}$ the result follows. It is useful to point out that the only reason for using $\tilde{A}$ instead of $A$ is to guarantee that $a_n(x)\in \R$ for every $x\in M$ so to fall in the setting of \cite[Proposition 4.9]{KS13}.
\end{proof}

Our next proposition tells us that under some conditions on the largest Lyapunov exponent the map $A$ can not be equal to the zero matrix in any point of the domain.

\begin{proposition}\label{prop: non-zero} 
Suppose there exist $\rho, \tau >0$ satisfying $\rho +\tau <\frac{\alpha\theta}{2}$ so that 
$$-\tau \leq \gamma_1(A,\mu)\leq \rho$$
for every ergodic $f$-invariant measure $\mu$ on $M$. Then,
$$A(x)\neq 0 \text{ for every } x\in M.$$
\end{proposition}

\begin{proof}
Suppose there exists $x\in M$ so that $A(x)=0$. Given $n\in \N$, let $p_n\in M$ be a periodic point associated to $\delta$, $S$ and
$$\{f^{-n}(x),\ldots,f^{-1}(x),x,f(x),\ldots,f^n(x)\}$$
by the periodic exponential specification property. In particular, $f^{2n+S}(p_n)=p_n$ and
$$d (f^{j} ( f^{-n}(p_n)),f^j (f^{-n}(x)))<\delta e^{-\theta \min\lbrace j, 2n-j\rbrace} \text{ for all } j=0,1,\ldots,2n.$$

Fix $\varepsilon >0$ so that $\rho+\tau+\varepsilon <\frac{\alpha \theta}{2}$ and let $C_\varepsilon >0$ be given by Proposition \ref{prop: cresc} associated to it. We start observing that, since $A(x)=0$,
\begin{displaymath}
\begin{split}
\|A^{2n+S}(p_n)\|&=\|A^{2n+S}(p_n)-A^{2n+S-1}(f(p_n))A(x)\|\\
&\leq \|A^{2n+S-1}(f(p_n))\|\|A(p_n)-A(x)\|\\
&\leq C_\varepsilon e^{(\rho+\varepsilon) (2n+S-1)}C_1 d(p_n,x)^\alpha \\
& \leq C_\varepsilon e^{(\rho+\varepsilon) (2n+S-1)}C_1 \delta^\alpha e^{-\theta \alpha n} \leq \hat{C} e^{(\rho +\varepsilon - \frac{\theta \alpha}{2})(2n+S) }
\end{split}
\end{displaymath}
where $\hat{C}=C_\varepsilon C_1 \delta^\alpha e^{\theta \alpha S} >0$ is independent of $n$. Thus, denoting by $\mu_{p_n}$ the ergodic $f$-invariant measure supported on the orbit of $p_n$ it follows that
\begin{displaymath}
\begin{split}
-\tau &\leq \gamma_1 (A,\mu _{p_n})=\lim_{k\to +\infty}\frac{1}{k}\log \|A^{k}(p_n)\| \\
&=\lim_{k\to +\infty}\frac{1}{k(2n+S)}\log \|A^{k(2n+S)}(p_n)\| \\
&=\lim_{k\to +\infty}\frac{1}{k(2n+S)}\log \|A^{(2n+S)}(p_n)^k\| \\
&\leq \lim_{k\to +\infty}\frac{1}{k(2n+S)}\log \|A^{(2n+S)}(p_n)\|^k \\
&=\frac{1}{(2n+S)}\log \|A^{(2n+S)}(p_n)\| \\
&\leq \rho +\varepsilon - \frac{\theta \alpha}{2} + \frac{\log \hat{C}}{(2n+S)}.
\end{split}
\end{displaymath}
In particular, 
$$0\leq \rho +\tau +\varepsilon - \frac{\theta \alpha}{2} + \frac{\log \hat{C}}{(2n+S)}.$$
Therefore, since $\rho+\tau +\varepsilon - \frac{\theta \alpha}{2}<0$ and $\hat{C}$ is independent of $n$, we get a contradiction whenever $n\gg 0$ concluding the proof of the proposition
\end{proof}

\subsection{Conclusion of the proof} Let $f:M\to M$ and $A:M\to M(d,\R)$ be as in Theorem \ref{theo:main}. To complete the proof of our main result the idea is to apply Proposition \ref{prop: non-zero} to the cocycle induced by $(A,f)$ on a suitable exterior power. In order to do so, for every $i \in \{1,\ldots , d\}$, let $\Lambda ^i(\mathbb{R}^d)$ be the $i$th exterior power of $\mathbb{R}^d$ which is the space of alternate $i$-linear forms on the dual $(\mathbb{R}^d)^*$ and $\Lambda ^iA(x):\Lambda ^i(\mathbb{R}^d) \to \Lambda ^i(\mathbb{R}^d)$ be the linear map given by 
\begin{equation}\label{eq: def induced map}
\Lambda ^iA(x) (\omega): (\phi _1, \ldots , \phi _i)\to \omega (\phi _1\circ A(x),\ldots ,\phi _i\circ A(x))
\end{equation}
for $\omega \in \Lambda ^i(\mathbb{R}^d) $ and $\phi _j\in (\mathbb{R}^d)^*$.
Then, the cocycle induced by $(A,f)$ on the $i$th exterior power is just the cocyle generated by $x\to \Lambda ^iA(x)$ over $f$. A very well known fact about this cocycle (see for instance \cite[Section 4.3.2]{LLE}) is that its Lyapunov exponents are given by
\begin{equation*} 
\{\gamma _{j_1}(A,\mu) +\ldots +\gamma _{j_i}(A,\mu); \; 1\leq j_1<\ldots < j_i\leq d\}.
\end{equation*}
In particular, its largest Lyapunov exponent is given by $\gamma_1(A,\mu)+\gamma _2(A,\mu)+\ldots+\gamma _i(A,\mu)$. Applying this fact to the case when $i=d$ we get that the \emph{only} Lyapunov exponent of $(\Lambda ^dA,f)$ with respect to $\mu$ is
$$\gamma_1(\Lambda ^dA,\mu)=\gamma_1(A,\mu)+\gamma _2(A,\mu)+\ldots+\gamma _d(A,\mu).$$
In particular,
$$-\tau \leq \gamma_1(\Lambda ^dA,\mu)\leq \rho$$
for every ergodic $f$-invariant measure $\mu$ on $M$. Thus, applying Proposition \ref{prop: non-zero} to the cocycle $(\Lambda ^dA,f)$ we conclude that $\Lambda ^dA(x)\neq 0$ for every $x\in M$. From \eqref{eq: def induced map} it follows then that $\text{ker}(A(x))=\{0\}$ for every $x\in M$. Consequently, $A(x)\in GL(d,\R)$ for every $x\in M$ concluding the proof of our main result.  \qed

\begin{remark}
As one can easily see from the proof, we have not used the full strength of the periodic specification property. Indeed, the following property on $f$ suffices: suppose there exist constants $\theta,\delta ,c,S>0$ so that for any point $x\in M$ and $n\in \N$ there exists a periodic point $p_n\in M$ such that $f^{k_n}(p_n)=p_n$ where $k_n\leq cn+S$ satisfying  
$$d (p_n,x)<\delta e^{-\theta  n}.$$
In this case, assuming there are $\rho, \tau >0$ satisfying $\rho +\tau <\frac{\alpha\theta}{c}$ so that 
$$-\tau \leq \gamma _d(A,\mu)+ \ldots + \gamma_1(A,\mu)\leq \rho$$
for every ergodic $f$-invariant measure $\mu$ on $M$, we get that $A(x)\in GL(d,\R)$ for every $x\in M$. The proof is the same, mutatis mutandis, as the one presented above. 

Note that if $f$ satisfies the periodic specification property then it satisfies the previous property with $c=2$. We chose to present the main result in terms of the periodic specification property because it is more recurrent in the literature.
\end{remark}


\medskip{\bf Acknowledgements.} The author was partially supported by a CAPES-Brazil postdoctoral fellowship under Grant No. 88881.120218/2016-01 at the University of Chicago. 


\end{document}